\documentclass{article}
\usepackage{amssymb,amsmath,amsthm,graphicx}

\def\utr{\, \underline{\triangleright}\, }
\def\otr{\, \overline{\triangleright}\, }
\def\utrd{\, \underline{\triangledown}\, }
\def\otrd{\, \overline{\triangledown}\, }

\def\B{\mathcal{B}}
\def\Aut{\operatorname{Aut}}
\def\Sym{\operatorname{Sym}}

\newtheorem{theorem}{Theorem}
\newtheorem{lemma}[theorem]{Lemma}

\newtheorem{corollary}[theorem]{Corollary}

\theoremstyle{definition}
\newtheorem{example}{Example}
\newtheorem{definition}{Definition}
\newtheorem{remark}{Remark}

\date{}

\title{\Large \textbf{Biracks and Switch Braid Quivers}}

\author{Max Chao-Haft\footnote{Email: mchaohaft@hmc.edu}\and
Sam Nelson\footnote{Email: Sam.Nelson@cmc.edu. 
Partially supported by Simons Foundation collaboration grant 702597.}}

\begin{document}
\maketitle

\begin{abstract}
We consider birack and switch colorings of braids. We define a switch structure
on the set of permutation representations of the braid group and
consider when such a representation is a switch automorphism. We define 
quiver-valued invariants of braids using finite switches and biracks
and use these to categorify the birack 2-cocycle invariant for braids. We 
obtain new polynomial invariants of braids via decategorification of
these quivers.
\end{abstract}

\parbox{6in} {\textsc{Keywords:} Braids, braid group representations, 
Burau representation,
biracks, quantum invariants, quiver 
enhancements, categorification, switches, 

\smallskip

}

\section{Introduction}

Initially defined in \cite{FRS}, \textit{biracks} and related structures 
known as \textit{switches} have been studied using 
several different notational conventions in works such as \cite{KR,FJK,EN}.
Biracks have the property that colorings of oriented framed knot, link or 
braid diagrams with birack elements satisfying a condition at the crossings
are preserved by Reidemeister II and III moves, making the number of colorings 
by a finite birack an easily-computable integer-valued 
invariant. Switches give colorings to semiarcs in braid diagrams
which are preserved by the braid Reidemeister moves, i.e. the Reidemeister 
III and direct Reidemeister II moves. 
These switch colorings give rise to a large family of permutation 
representations of the braid group, as discussed in \cite{FJK}.
Moreover, the preservation of colorings under moves 
means that any invariant of birack- or switch-colored knots or braids can be 
used to define generally stronger invariants known as \textit{enhancements}.

In this paper we introduce a new family of braid group representations 
derived from switch colorings, called \textit{switch braid representations}. 
These are precisely the permutation representations from \cite{FJK}
which restrict to representations in the category of switches. We show that 
they are equivalently the representations arising from \textit{medial} 
switches, and that they generalize the classical Burau representations.

We also introduce a finite quiver-valued invariant of braids using 
finite biracks and switches. Since finite quivers are small categories, this 
construction categorifies the birack and switch counting invariants for 
braids. The quivers 
we obtain, like those in \cite{CN}, decategorifiy to give us new polynomial 
invariants of braids. Moreover, the set of these invariant quivers has an
algebraic structure of its own which may prove to be of future interest.
Weighting these quivers with 2-cocycles in the birack case yields a 
quiver categorification of the birack 2-cocycle invariant for braids. We are 
then obtain an infinite family of new two-variable polynomial invariants of 
braids by decategorifying these weighted quivers.

The paper is organized as follows. In Section \ref{BnB} we recall the basics of
biracks and birack colorings. In Section \ref{SBR} we
recall the basics of switches and introduce our new 
switch braid representations, characterizing them as the representations 
induced by  
medial switches and showing that they specialize to the Burau representations.
In Section \ref{BB} we define our new switch braid quivers and a new polynomial 
invariant of braids, the switch braid quiver polynomial, associated to each finite 
switch. In Section \ref{BCQ} we enhance the switch braid quiver invariants with 
birack cocycles and obtain two new families of 2-variable polynomial invariants
of braids as decategorifications of the birack braid cocycle quiver. 
We illustrate the new invariants with examples. We conclude in 
Section \ref{Q} with some questions for future research.

\section{Biracks and Braids}\label{BnB}

We begin this section with a definition. See \cite{EN} for more.

\begin{definition}\label{birack}
Let $X$ be a set. A \textit{birack structure} on $X$ is a pair of binary
operations $\utr,\otr:X\times X\to X\times X$ satisfying the following conditions:
\begin{itemize}
\item[(i)] For all $y\in X$ the maps $\alpha_y,\beta_y:X\to X$ 
defined by 
\[
\alpha_y(x)=x\otr y \quad \mathrm{and} \quad \beta_y(x)=x\utr y
\]
are bijective.
\item[(ii)] The map
$S:X\times X\to X\times X$ defined by
\[
S(x,y)=(y\otr x, x\utr y)
\]
is bijective.
\item[(iii)] For all $x,y,z\in X$ the \textit{exchange laws} are satisfied:
\[
\begin{array}{rcl}
(x\utr y)\utr (z\utr y) & = &(x\utr z)\utr (y\otr z) \\
(x\utr y)\otr (z\utr y) & = &(x\otr z)\utr (y\otr z) \\
(x\otr y)\otr (z\otr y) & = &(x\otr z)\otr (y\utr z). 
\end{array}
\]
\end{itemize}
A birack structure additionally satisfying $x\utr x=x\otr x$ for all $x\in X$
is called a \textit{biquandle structure}. We refer to the 
triple $(X,\utr,\otr)$ (or just the set $X$) as a \textit{birack} or \textit{biquandle} if 
the operations $\utr, \otr$ define a birack structure or biquandle structure, 
respectively. 
\end{definition}

\begin{example}
Standard examples of biracks include
\begin{itemize}
\item \textit{Constant action biracks.} For any set $X$ and commuting 
bijections $\sigma,\tau:X\to X$, the operations $x\utr y=\sigma(x)$ and 
$x\otr y=\tau(x)$ define a birack structure on $X$. This is called a
\textit{constant action birack} because, for a given element $x$, 
the actions $y \mapsto x \utr y:X \to X$ and 
$y \mapsto x\otr y:X \to X$ are constant.
If $\sigma=\tau$, then the birack is a biquandle;
if $\sigma=\tau = \mathrm{Id}$ is the identity map, then the birack is called 
\textit{trivial}.
\item \textit{Alexander biquandles.} For any module $X$ over the ring 
$\mathbb{Z}[t^{\pm 1},s^{\pm 1}]$ of two-variable Laurent polynomials, the
operations $x\utr y=tx+(s-t)y$ and $x\otr y=sx$ define a biquandle structure
known as an \textit{Alexander biquandle}. The special case $s=1$ yields the
classical Alexander module.
\item Let $X$ have two (not necessarily abelian) group operations $*$ 
and $\circ$ satisfying the modified distributivity condition
\[x\circ(y*z)=(x\circ y)*x^**(x\circ z)\]
where $x^*$ is the inverse of $x$ with respect to the $*$ operation.
Then $X$ is a \textit{skew brace}, with birack operations given by 
$x\utr y=y^{\circ}\circ(x*y)$ and $x\otr y=y^{\circ}\circ(y*x)$ where
$y^{\circ}$ is the inverse of $y$ with respect to the $\circ$ operation.
\end{itemize}
\end{example}

We will specify finite birack structures by listing their operation
tables. For example, the smallest nontrivial birack has underlying set
$X=\{1,2\}$ and operation table
\[\begin{array}{r|rr} \utr & 1 & 2 \\ \hline 1 & 2 & 2 \\ 2 & 1 & 1 \end{array}
\quad
\begin{array}{r|rr} \otr & 1 & 2 \\ \hline 1 & 2 & 2 \\ 2 & 1 & 1 \end{array}~.
\]

The birack axioms are chosen 
so that the number of \textit{birack colorings}
of an oriented knot, link or braid diagram is invariant under the
Reidemeister II and III moves. 
A \textit{birack coloring} of a diagram $D$ by a birack $X$
(also called an \textit{$X$-coloring} of $D$) 
is an assignment of an element of $X$ to each of the semiarcs of $D$ (i.e., sections
of the diagram between crossing points) such that at every crossing,
we have the following relationship:
\[\includegraphics{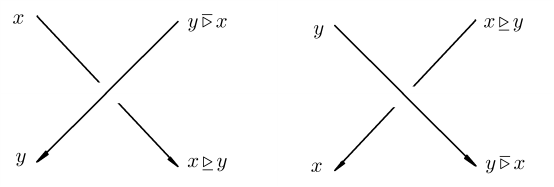}\]

It is then a standard exercise to check that the birack axioms
[(i), (ii), (iii) in Definition \ref{birack}] are precisely the
conditions required by the Reidemeister II and III moves to guarantee that for any 
birack-colored oriented knot (or link or braid) diagram before a move, there is
a unique birack-colored knot (or link or braid) diagram after the move which agrees with the
initial diagram outside the neighborhood of the move. 
In particular, 
\begin{itemize}
\item axiom (i) corresponds to the direct Reidemeister II moves,
\item axiom (ii) corresponds to the reverse Reidemeister II moves,
\item axiom (iii) corresponds to the all-positive Reidemeister III 
move. 
\end{itemize}
Combined with the 
direct Reidemeister II moves, this gives us invariance under all of the Reidemeister III moves.

Let us write $\mathcal{C}(D,X)$ for the set of all $X$-colorings of $D$. 
Biracks have traditionally been applied to 
study knots and links in the following manner: 

\begin{definition}
Given a knot or link diagram $K$ and a finite birack $X$, the 
\textit{birack counting invariant} $\Phi_X^{\mathbb{Z}}(K)$ is defined as the 
cardinality of the set $\mathcal{C}(K,X)$.
\end{definition}

It is a standard exercise to check that if two knot or link
diagrams $K,K'$  are related by Reidemeister II and III moves, then
\[\Phi_X^{\mathbb{Z}}(K) = |\mathcal{C}(K,X)| = |\mathcal{C}(K',X)| = \Phi_X^{\mathbb{Z}}(K').\]

Consequently, we have:

\begin{corollary}\label{birack_counting}
The birack counting invariant is an invariant of knots and links. 
\end{corollary}

We could define a similar counting invariant for braids. 
However, taken naively, the number of birack colorings 
of a braid diagram is a trivial invariant.
Indeed, at any crossing, the colors of the 
top two semiarcs determine 
the colors of the bottom two semiarcs
\[
\includegraphics[scale=.5]{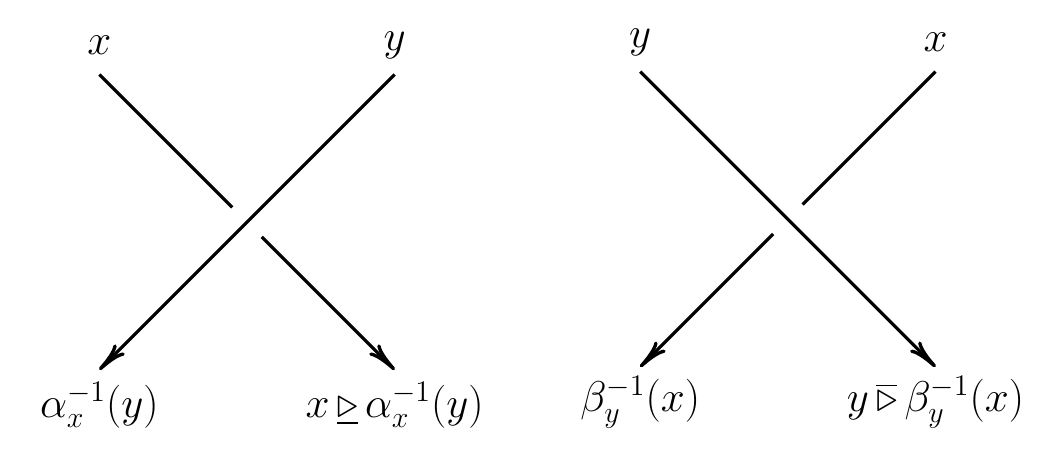}
\]
so it follows that the vector of colors $(x_1, \dots, x_n)$ at the top of the 
braid diagram uniquely determines the rest of the coloring. 
Moreover, any vector of colors $(x_1, \dots, x_n)\in X^n$ 
at the top will yield a 
valid coloring. 
Hence, 
the number colorings of an $n$-strand braid diagram $B$
by a finite birack $X$ is simply $|X|^n$, and 
thus is
completely determined by the size of $X$ and the number of strands of $B$.
Thus we will need a different approach to obtain fruitful braid invariants 
from biracks.

\section{Switch Braid Representations}\label{SBR}

%
%

When dealing with braids, the only relevant Reidemeister moves are 
the Reidemeister III moves and the direct Reidemeister II moves, which we 
will call the \textit{braid Reidemeister moves}. 
The reverse Reidemeister II moves are unnecessary because 
braid strands are always oriented in the same direction. 
Hence, we may allow colorings by a
slightly more general class of structures than biracks, 
which are only required to
respect the Reidemeister III moves and direct Reidemeister II moves. 
Such structures are defined as follows.

\begin{definition}\label{switch}
A \textit{switch structure} on a set $X$ is an invertible map 
$\rho: X \times X \to X \times X$ satisfying 
the \textit{set-theoretic Yang-Baxter equation}:
\[
(\rho \times \mathrm{Id}_X)(\mathrm{Id}_X \times \rho)(\rho \times \mathrm{Id}_X) = 
(\mathrm{Id}_X \times \rho)(\rho \times \mathrm{Id}_X)(\mathrm{Id}_X \times \rho).
\]
We refer to the pair $(X,\rho)$ (or just the set $X$) as a \textit{switch}. 
\end{definition}

\begin{remark}
\label{birack_is_a_switch}
Any birack $(X,\utr, \otr)$ can be realized as a switch $(X,\rho)$ 
by defining $\rho(x,y) = (\alpha_x^{-1}(y), x \utr \alpha_x^{-1}(y))$. 
\end{remark}

A \textit{switch coloring} of a braid diagram $B$ by a switch $(X,\rho)$, 
or an \textit{$X$-coloring} of $B$, is an assignment of an element of $X$ 
to each semiarc of $B$ such that, at every crossing, we 
have the following relationship: 
\[
\includegraphics[scale=.5]{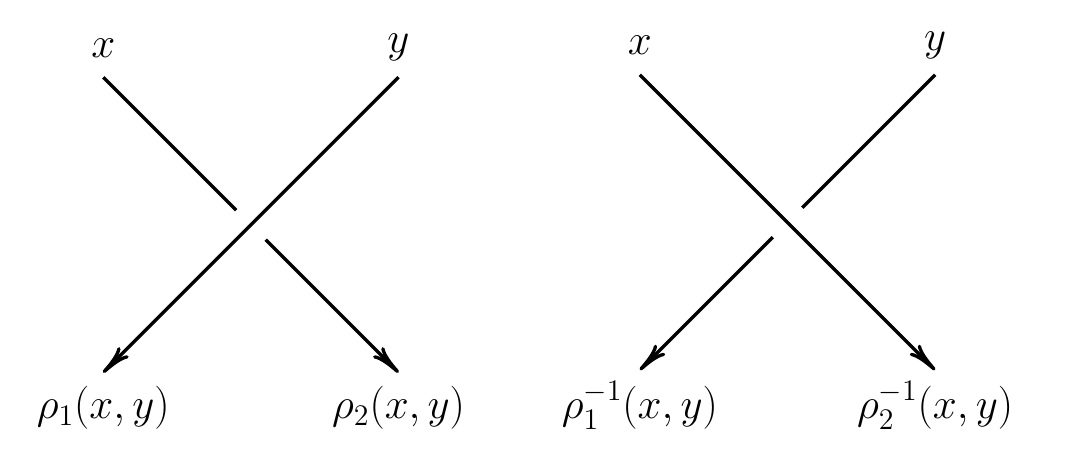}
\]
where $\rho_i^{\pm 1}(x,y)$ is the $i$th coordinate of $\rho^{\pm 1}(x,y)$. 
Note that switch colorings are top-down, as opposed to the left-right birack 
colorings described in Section \ref{BnB}.
Just as with biracks, any vector of colors $(x_1, \dots, x_n)\in X^n$ 
at the top of the braid diagram determines a unique $X$-coloring. 

Given a color vector $\overline{x} = (x_1, \dots, x_n) \in X^n$, let us write 
$\overline{x}\cdot B$ for the color vector $(y_1, \dots, y_n)\in X^n$ induced 
at the bottom of $B$ upon coloring the top of $B$ with $\overline{x}$.
This gives us a map 
\[
(\overline{x},B) \to \overline{x}\cdot B \;: \;X^n\times \{\text{diagrams of $n$-strand braids}\} \to X^n.
\]
It is a standard exercise to check that for any $\overline{x}\in X^n$ and any 
braid diagrams $B,B'$ related by braid 
Reidemeister moves, $\overline{x}\cdot B = \overline{x}\cdot B'$. 
Hence, we have:
\begin{theorem}
The function 
\[
\varphi: \{\text{diagrams of $n$-strand braids}\} \to \left(X^n\right)^{X^n}
\] 
sending $B \mapsto \varphi_B$, in which 
$\varphi_B: X^n \to X^n$ is defined by 
$\varphi_B(\overline{x}) = \overline{x}\cdot B$, is an invariant of 
braids.
\end{theorem}

\begin{remark}
Note that we can treat $\varphi$ as a function $\B_n \to (X^n)^{X^n}$
(with domain the \textit{$n$-strand braid group}). 
We will do this going forward.
\end{remark}

Furthermore, the map $(\overline{x},B) \to \overline{x}\cdot B: X^n \times \B_n \to X^n$ 
defines a right group action of $\B_n$ on $X^n$. Indeed, the identity braid $1_n \in \B_n$ 
satisfies $\overline{x}\cdot 1_n = \overline{x}$ for every $\overline{x}\in X^n$. And, 
since the group operation in $\B_n$ is vertical stacking of braids, 
we have $\overline{x}\cdot (B_1B_2) = (\overline{x} \cdot B_1) \cdot B_2$ 
for all $\overline{x}\in X^n$ and $B_1,B_2\in \B_n$. Consequently, each 
$\varphi_B$ is in $\Sym(X^n)$, the automorphism group of the set $X^n$, and 
moreover: 

\begin{theorem}
$\varphi: \B_n \to \Sym(X^n)$ is a contravariant functor.
\end{theorem}

The functor $\varphi$ gives us a representation of $\B_n$ 
in the category of bijections (set-automorphisms) of $X^n$. 
This much is known and well-documented in the literature. See, for 
example, \cite{FJK}.

However, $X$ is not merely a set: it is endowed with the additional structure 
of a  switch. It is natural to ask whether the switch structure on $X$ 
induces a switch structure on $X^n$ and, if so, whether $\varphi$ maps into 
the category of switch automorphisms of $X^n$. We will show that the answer 
to the first question is ``yes!'' and 
the answer to the second is ``sometimes.'' 

\subsection{Switch Structure on $X^n$}

To make our lives easier, we will first move 
to a universal-algebra definition of switch. 
Given a switch structure $\rho$ on a set $X$, we can define binary operations 
$\utrd, \otrd, \utrd^{-1}, \otrd^{-1}: X \times X \to X$ by 
\[
x \utrd y = \rho_2(x,y), \quad x \otrd y = \rho_1(y,x), \quad x \utrd^{-1} y = \rho_1^{-1}(y,x), 
\quad x \otrd^{-1} y = \rho_2^{-1}(x,y).
\]
Elsewhere in the literature, $x \utrd y, \,x \otrd y,\, x \utrd^{-1} y,\, x \otrd^{-1} y$ 
are written $x^y, \,x_y,\, x^{\overline{y}}, \,x_{\overline{y}}$, respectively (see \cite{FJK}).

The fact that $\rho$ is a switch structure ensures that
\begin{itemize}

\item[(i)] for all $x,y,z\in X$, 
\begin{align*}
(z \otrd(x \utrd y)) \otrd (y \otrd x) &= (z \otrd y) \otrd x\\
(y \otrd x) \utrd (z \otrd (x \utrd y)) &= (y \utrd z) \otrd (x \utrd (z \otrd y))\\
(x \utrd y) \utrd z &= (x \utrd (z \otrd y)) \utrd (y \utrd z)
\end{align*}

\item[(ii)] for all $x,y\in X$, 
\begin{align*}
(x \utrd y) \utrd^{-1} (y \otrd x) &= x = (x \otrd^{-1}) \otrd (y \utrd^{-1} x)\\
(y \otrd x) \otrd^{-1} (x \utrd y) &= y = (y \utrd^{-1} x) \utrd (x \otrd^{-1} y).
\end{align*}
\end{itemize}

Conversely, given a 4-tuple $(\utrd, \otrd, \utrd^{-1}, \otrd^{-1})$ of binary 
operations on $X$ satisfying (i) and (ii), we can define functions $\rho, \rho': X \times X \to X \times X$ by 
\[
\rho(x,y) = \left(y \otrd x, x \utrd y \right) \quad \text{and} 
\quad \rho'(x,y) = \left(y \utrd^{-1} x, x \otrd^{-1} y \right).
\]
Then (i) ensures that $\rho$ satisfies the Yang-Baxter equation and 
(ii) ensures that $\rho\rho' = \rho'\rho = \mathrm{Id}_{X \times X}$. 
Hence, $\rho$ defines a switch structure on $X$. 

We have now described 
two maps:
\[
\left\{\text{switch structures on $X$}\right\} \rightleftharpoons
\left\{\text{tuples $(\utrd, \otrd, \utrd^{-1}, \otrd^{-1})$ satisfying (i) and (ii)}\right\}.
\]
One can see that these maps are inverses of each other, thus proving the 
following theorem:

\begin{theorem}\label{equivalent}
Let $X$ be a set. Then there is a bijective correspondence 
\[
\{\emph{switch structures on $X$}\} \longleftrightarrow
\left\{\emph{tuples $(\utrd, \otrd, \utrd^{-1}, \otrd^{-1})$ satisfying (i) and (ii)}\right\}.
\]
\end{theorem}

Moreover, switch colorings are the same as colorings in which the following relationships 
hold at crossings:
\[
\includegraphics[scale=.5]{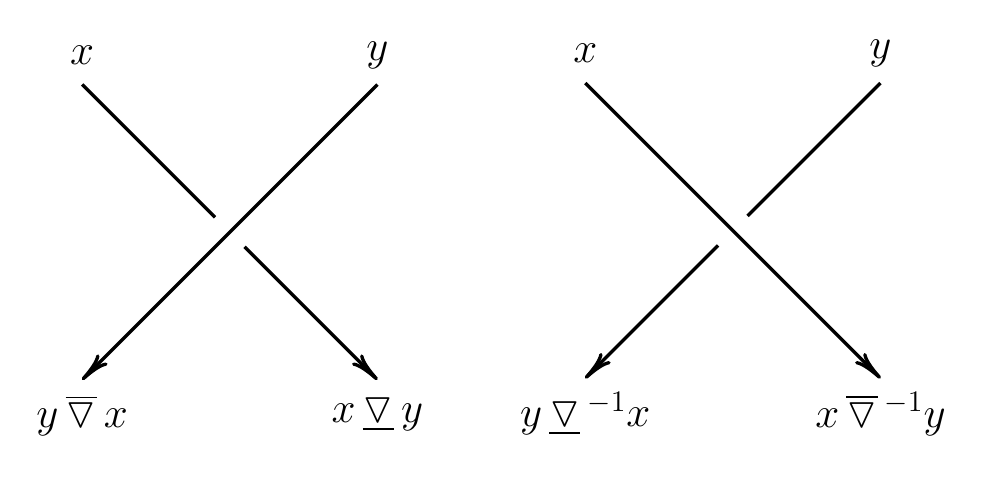}
\]

Henceforth, we will use the term \textit{switch} to refer to 
either of the two equivalent structures discussed in Theorem \ref{equivalent}. 
Now we are ready to define switch products.
Given a family of switches $\left\{\left(X_i, \utrd_i, \otrd_i, \utrd^{-1}_i, \otrd^{-1}_i\right)\right\}_{i \in I}$, 
we define binary operations $\utrd, \otrd, \utrd^{-1}, \otrd^{-1}$ on $\prod_{i\in I} X_i$ as follows:
\begin{align*}
\overline{x} \utrd \overline{y} &= \left( x_i \,\utrd_i \,y_i \right)_{i \in I}\\
\overline{x} \otrd \overline{y} &= \left( x_i \,\otrd_i \,y_i \right)_{i \in I}\\
\overline{x} \utrd^{-1} \overline{y} &= \left( x_i \,\utrd_i^{-1} \,y_i \right)_{i \in I}\\
\overline{x} \otrd^{-1} \overline{y} &= \left( x_i \,\otrd_i^{-1} \,y_i \right)_{i \in I}.
\end{align*}
Since the operations are defined pointwise, it follows immediately that:

\begin{theorem}
$\left(\prod_{i\in I} X_i, \utrd, \otrd, \utrd^{-1}, \otrd^{-1} \right)$ is a switch.
\end{theorem}

In particular, any switch structure on $X$ induces a natural 
switch structure on $X^n$. Hence, for every $B\in \B_n$, the map $\varphi_B$ is a bijection of 
a switch. But when is $\varphi_B$ a \textit{switch automorphism}?

\subsection{When $\varphi_B$ is a Switch Automorphism}

We begin with a definition.

\begin{definition}
A \textit{switch homomorphism} from $(X,\rho)$ to $(Y,\tau)$ is a 
map $f:X \to Y$ satisfying 
\[
\tau(f(x),f(y)) = \left(f(\rho_1(x,y)), f(\rho_2(x,y))\right)
\]
for all $x,y\in X$.
\end{definition}

Equivalently, a switch homomorphism from $\left(X, \utrd_1, \otrd_1, \utrd^{-1}_1, \otrd^{-1}_1\right)$
to $\left(X, \utrd_2, \otrd_2, \utrd^{-1}_2, \otrd^{-1}_2\right)$ is a map 
$f:X \to Y$ satisfying 
\[
f(x \,\utrd_1 \,y ) = f(x) \,\utrd_2 \,f(y) \quad \text{and} \quad 
f(x \, \otrd_1 \,y) = f(x) \,\otrd_2 \,f(y)
\]
for all $x,y\in X$. One can readily check that this definition yields a category 
whose objects are switches and whose morphisms are switch homomorphisms. 
Furthermore, one can check that the isomorphisms in this category are precisely the 
morphisms whose underlying maps are bijective (in other words, the inverse of a 
bijective switch homomorphism is also a switch homomorphism). 
Thus, $\varphi_B : X^n \to X^n$ is a switch automorphism 
if and only if it is a switch homomorphism. 

So it suffices to characterize when $\varphi_B$ is a switch homomorphism. 
This is the main result of the section:

\begin{theorem}\label{homomorphism_iff_abelian}
$\varphi_B:X^n \to X^n$ is a switch homomorphism for every 
$B\in \B_n$ if and only if $X$ is medial. 
\end{theorem}

The definition of \textit{medial switch} is a special case of 
the more general concept of \textit{entropic variety} 
from universal algebra (see \cite{DD} for more).
It generalizes the concepts of \textit{medial quandle} and 
\textit{medial biquandle} 
studied in works such as \cite{JPSZ,SCN,SCH}.

\begin{definition}\label{medial}
A \textit{medial switch} (also called an \textit{entropic switch} or 
\textit{abelian switch}) is a switch $X$ such that 
\begin{align*}
(x \utrd y) \utrd (w \utrd z) &= (x \utrd w) \utrd (y \utrd z)\\
(x \utrd y) \otrd (w \utrd z) &= (x \otrd w) \utrd (y \otrd z)\\
(x \otrd y) \otrd (w \otrd z) &= (x \otrd w) \otrd (y \otrd z)
\end{align*}
for all $x,y,w,z\in X$.
\end{definition}

\begin{lemma}\label{medial_lemma}
A switch $X$ is medial if and only if $\varphi_{\sigma_i}:X^n \to X^n$
is a switch homomorphism for every generator $\sigma_i$ of $\B_n$.
\end{lemma}

As usual, $\sigma_1, \dots, \sigma_{n-1}$ 
denote the generators in the Artin presentation of $\B_n$.

\begin{proof}
The lemma is most easily understood by studying the following diagrams:
\[\includegraphics[scale=.375]{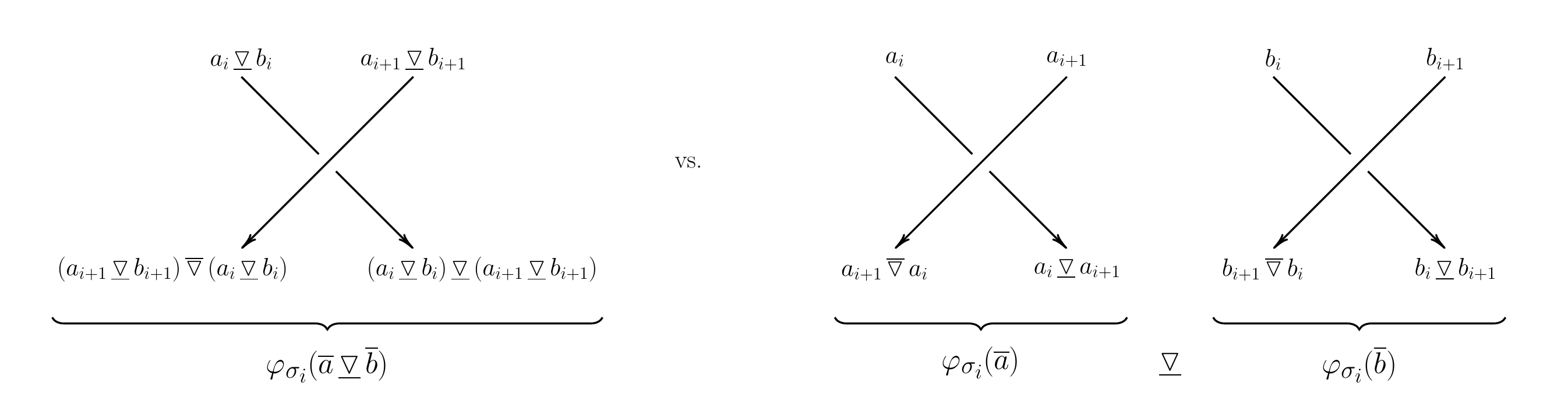}\]
\[\includegraphics[scale=.375]{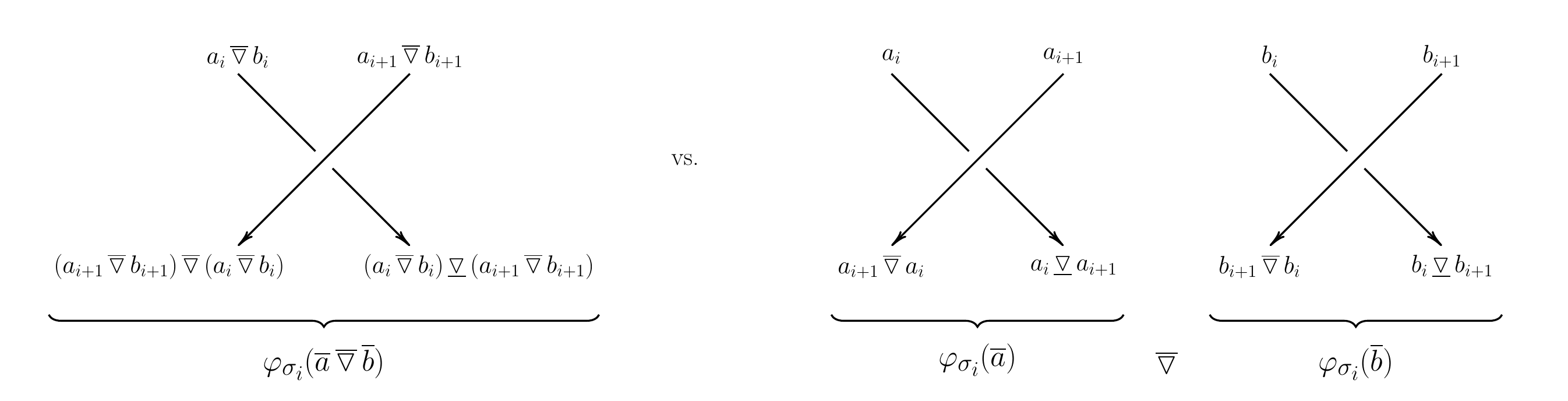}\]
In more detail, we argue as follows:
\begin{itemize}

\item[($\Rightarrow$):] Suppose $X$ is medial. Let $i\in \{1, \dots, n-1\}$
and  
$\overline{a}, \overline{b} \in X^n$. 
Then for all $k\in \{1, \dots, n\}$, we have
\begin{align*}
\left[\varphi_{\sigma_i}(\overline{a} \utrd \overline{b})\right]_k 
&= 
\begin{cases}
(a_{i+1} \utrd b_{i+1}) \otrd (a_i \utrd b_i) &\text{if $k=i$}\\
(a_{i} \utrd b_{i}) \utrd (a_{i+1} \utrd b_{i+1}) &\text{if $k=i+1$}\\
a_k \utrd b_k &\text{otherwise}
\end{cases}\\
&=
\begin{cases}
(a_{i+1} \otrd a_{i}) \utrd (b_{i+1} \otrd b_i) &\text{if $k=i$}\\
(a_{i} \utrd a_{i+1}) \utrd (b_{i} \utrd b_{i+1}) &\text{if $k=i+1$}\\
a_k \utrd b_k &\text{otherwise}
\end{cases}\\
&=
\left[\varphi_{\sigma_i}(\overline{a}) \utrd \varphi_{\sigma_i}(\overline{b}) \right]_k
\end{align*}
and 
\begin{align*}
\left[\varphi_{\sigma_i}(\overline{a} \otrd \overline{b})\right]_k 
&= 
\begin{cases}
(a_{i+1} \otrd b_{i+1}) \otrd (a_i \otrd b_i) &\text{if $k=i$}\\
(a_{i} \otrd b_{i}) \utrd (a_{i+1} \otrd b_{i+1}) &\text{if $k=i+1$}\\
a_k \utrd b_k &\text{otherwise}
\end{cases}\\
&=
\begin{cases}
(a_{i+1} \otrd a_{i}) \otrd (b_{i+1} \otrd b_i) &\text{if $k=i$}\\
(a_{i} \utrd a_{i+1}) \otrd (b_{i} \utrd b_{i+1}) &\text{if $k=i+1$}\\
a_k \utrd b_k &\text{otherwise}
\end{cases}\\
&=
\left[\varphi_{\sigma_i}(\overline{a}) \otrd \varphi_{\sigma_i}(\overline{b}) \right]_k.
\end{align*}
Hence, 
\[
\varphi_{\sigma_i}(\overline{a} \,\utrd \,\overline{b} ) = 
\varphi_{\sigma_i}(\overline{a}) \,\utrd \,\varphi_{\sigma_i}(\overline{b}) 
\quad\text{and}\quad
\varphi_{\sigma_i}(\overline{a} \, \otrd \,\overline{b}) = 
\varphi_{\sigma_i}(\overline{a}) \,\otrd \,\varphi_{\sigma_i}(\overline{b})
\]
Therefore, $\varphi_{\sigma_i}$ is a switch homomorphism
for every braid group generator $\sigma_i$.

%

\item[($\Leftarrow$):] Suppose $\varphi_{\sigma_i}:X^n \to X^n$
is a switch homomorphism for every 
generator $\sigma_i$ of $\B_n$. 
Let $x,y,w,z\in X$. Choose $i \in \{1, \dots, n-1\}$ and $\overline{a}, \overline{b}\in X^n$ so that 
\[
a_i = x, \quad a_{i+1} = w, \quad b_i = y, \quad b_{i+1} = z.
\]
Then, since $\varphi_{\sigma_i}$ is a switch homomorphism, we have 
\begin{equation}\label{eq:utrd}
\varphi_{\sigma_i}(\overline{a} \,\utrd \,\overline{b} ) = 
\varphi_{\sigma_i}(\overline{a}) \,\utrd \,\varphi_{\sigma_i}(\overline{b}) 
\end{equation}
and
\begin{equation}\label{eq:otrd}
\varphi_{\sigma_i}(\overline{a} \, \otrd \,\overline{b}) = 
\varphi_{\sigma_i}(\overline{a}) \,\otrd \,\varphi_{\sigma_i}(\overline{b}).
\end{equation}
Taking the ($i+1$)th components of Equation \ref{eq:utrd}, we obtain
\begin{align*}
(x \utrd y) \utrd (w \utrd z) &= (a_i \utrd b_i) \utrd (a_{i+1} \utrd b_{i+1})\\
&= \left[\varphi_{\sigma_i}(\overline{a} \utrd \overline{b})\right]_{i+1}\\
&= \left[\varphi_{\sigma_i}(\overline{a}) \utrd \varphi_{\sigma_i}(\overline{b})\right]_{i+1}\\
&= (a_i \utrd a_{i+1}) \utrd (b_i \utrd b_{i+1})\\
&= (x \utrd w) \utrd (y \utrd z).
\end{align*}
Taking the $i$th components of Equation \ref{eq:utrd}, we obtain 
\begin{align*}
(w \utrd z) \otrd (x \utrd y) &= (a_{i+1} \utrd b_{i+1}) \otrd (a_{i} \utrd b_{i})\\
&= \left[\varphi_{\sigma_i}(\overline{a} \utrd \overline{b})\right]_{i}\\
&= \left[\varphi_{\sigma_i}(\overline{a}) \utrd \varphi_{\sigma_i}(\overline{b})\right]_{i}\\
&= (a_{i+1} \otrd a_{i}) \utrd (b_{i+1} \otrd b_{i})\\
&= (w \otrd x) \utrd (z \otrd y).
\end{align*}
Taking the $i$th components of Equation \ref{eq:otrd}, we obtain
\begin{align*}
(w \otrd z) \otrd (x \otrd y) &= (a_{i+1} \otrd b_{i+1}) \otrd (a_{i} \otrd b_{i})\\
&= \left[\varphi_{\sigma_i}(\overline{a} \otrd \overline{b})\right]_{i}\\
&= \left[\varphi_{\sigma_i}(\overline{a}) \otrd \varphi_{\sigma_i}(\overline{b})\right]_{i}\\
&= (a_{i+1} \otrd a_{i}) \otrd (b_{i+1} \otrd b_{i})\\
&= (w \otrd x) \otrd (z \otrd y).
\end{align*}
Therefore, $X$ is medial.

%

\end{itemize}
This proves the lemma.
\end{proof}

\begin{proof}[Proof of Theorem \ref{homomorphism_iff_abelian}]
By Lemma \ref{medial_lemma}, it suffices to show that $\varphi_B:X^n \to X^n$ is a switch homomorphism for every $B\in \B_n$
if and only if $\varphi_{\sigma_i}:X^n \to X^n$
is a switch homomorphism for every 
generator $\sigma_i$ of $\B_n$. 
\begin{itemize}

\item[($\Rightarrow$):] This direction is trivial.

\item[($\Leftarrow$):] Suppose $\varphi_{\sigma_i}:X^n \to X^n$
is a switch homomorphism for every 
generator $\sigma_i$ of $\B_n$.
Let $B\in \B_n$. Then we can write 
$B = \sigma_{\iota(1)}^{\epsilon(1)}\dots \sigma_{\iota(k)}^{\epsilon(k)}$ 
for some $k$ and functions $\iota: \{1, \dots, k\} \to \{1, \dots, n-1\}$ and
$\epsilon: \{1, \dots, k\} \to \{+,-\}$. Then, using the fact that $\varphi$ is a
contravariant functor, we obtain
\begin{align*}
\varphi_B &= \varphi(B)\\ &=
\varphi\left(\sigma_{\iota(1)}^{\epsilon(1)}\dots \sigma_{\iota(k)}^{\epsilon(k)}\right)\\
&= \varphi\left(\sigma_{\iota(k)}^{\epsilon(k)}\right)\dots\varphi \left(\sigma_{\iota(1)}^{\epsilon(1)}\right)\\
&= \varphi\left(\sigma_{\iota(k)}\right)^{\epsilon(k)}\dots\varphi \left(\sigma_{\iota(1)}\right)^{\epsilon(1)}\\
&=  \varphi_{\sigma_{\iota(k)}}^{\epsilon(k)}\dots\varphi_{\sigma_{\iota(1)}}^{\epsilon(1)}.
\end{align*}
Since each $\varphi_{\sigma_{\iota(i)}}$ is a switch homomorphism, it follows that $\varphi_B$ is a 
switch homomorphism. 
\end{itemize}
This completes the proof.
\end{proof}

For a switch $X$, let $\Aut(X^n)$ denote 
the automorphism group of the switch $X^n$.
We then have the following:

\begin{corollary}
Let $X$ be a switch. Then 
\[
 \text{$\varphi$ defines a functor 
$\B_n \to \Aut(X^n)$} \iff \text{$X$ is medial}.
\]
\end{corollary}

\begin{definition}
Let $X$ be a medial switch. Then the contravariant functor $\varphi: \B_n \to \Aut(X^n)$
is called a \textit{switch braid representation}.
\end{definition}

This defines an infinite family of braid group representations.

%

\subsection{Examples}

\begin{example}\label{alexander_switch}
Let $X$ be a module over a commutative ring $R$ and 
let $\lambda, \mu$ be invertible elements of $R$. The map 
$\rho: X \times X \to X\times X$ given by 
\[
\rho(x,y) = \big(\lambda y + (1-\mu\lambda)x,\; \mu x\big)
\]
defines a switch structure on $X$ called the \textit{Alexander switch}.
The $\utrd, \otrd$ operations are given by
\[
x \utrd y = \mu x \quad \text{and} \quad x \otrd y = \lambda x + (1-\mu\lambda)y.
\]
Observe that for all
$x,y,w,z\in X$, we have 
\begin{align*}
(x \utrd y) \utrd (w \utrd z) 
&= (\mu x ) \utrd (w \utrd z)\\
&=  \mu^2 x\\
&= (\mu x) \utrd (y \utrd z)\\
&= (x \utrd w) \utrd (y \utrd z)
\end{align*}
and 
\begin{align*}
(w \utrd z) \otrd (x \utrd y) 
&= (\mu w) \otrd (\mu  x)\\
&= \lambda \mu w + (1-\mu \lambda) \mu x\\
&= \mu (\lambda w + (1-\mu\lambda)x)\\
&= (\lambda w + (1-\mu\lambda)x) \utrd (z \otrd y)\\
&= (w \otrd x) \utrd (z \otrd y)
\end{align*}
and
\begin{align*}
(w \otrd z) \otrd (x \otrd y) 
&= (\lambda w + (1-\mu\lambda) z) \otrd (\lambda x +(1-\mu\lambda) y)\\
&= \lambda (\lambda w + (1-\mu\lambda) z) + (1-\mu\lambda) (\lambda x +(1-\mu\lambda) y)\\
&= \lambda^2 w + \lambda(1-\mu\lambda) z + \lambda(1-\mu\lambda) x + (1-\mu\lambda)^2 y\\
&= \lambda (\lambda w + (1-\mu\lambda) x) + (1-\mu\lambda) (\lambda z +(1-\mu\lambda) y)\\
&= (\lambda w + (1-\mu\lambda) x) \otrd (\lambda z +(1-\mu\lambda) y)\\
&= (w \otrd x) \otrd (z \otrd y).
\end{align*}
Therefore, the Alexander switch is medial. Consequently, 
each Alexander switch gives 
rise to a switch braid representation of the braid group. 
\end{example}

\begin{example}
In this example, we show that the classical 
Burau representation of the braid group is a special 
case of the switch braid representation. 
Let $\Lambda = \mathbb{Z}[t,t^{-1}]$ be the ring 
of Laurent polynomials. 
The \textit{contravariant Burau representation} of $\B_n$ is the unique group 
homomorphism 
$\psi_n : (\B_n)^{\mathsf{op}} \to \mathrm{GL}_n(\Lambda)$ sending
\[
\sigma_i \mapsto U_i=
\begin{pmatrix}
I_{i-1} & 0 & 0 & 0\\
0 & 1-t & t & 0\\
0 & 1 & 0 & 0\\
0 & 0 & 0 & I_{n-i-1}
\end{pmatrix}.
\]
First observe that $\Lambda$ can be viewed as a 
module over itself. 
We can then
define an
Alexander switch structure on $\Lambda$ by taking
$\rho: \Lambda \times \Lambda \to \Lambda \times \Lambda$ to be
\[
\rho(x,y) = \big( ty + (1-t)x, \; x \big).
\]
Here we have chosen our invertible elements to be 
$\lambda = t$ and $\mu = 1$. 
Then by Example
\ref{alexander_switch}, $\Lambda$ defines a medial switch. 
Moreover, the switch braid representation 
$\varphi : \B_n \to \Aut(\Lambda^n)$ arising from 
$\Lambda$ satisfies the following: for all generators 
$\sigma_i$ of $\B_n$ and all $\overline{a} \in \Lambda^n$, 
\begin{align*}
\varphi_{\sigma_i}(\overline{a}) 
&= 
\begin{pmatrix}
a_1&
\dots&
a_{i-1}&
a_{i+1} \otrd a_i&
a_i \utrd a_{i+1}&
a_{i+2}&
\dots&
a_n
\end{pmatrix}^\intercal\\
&= 
\begin{pmatrix}
a_1&
\dots&
a_{i-1}&
ta_{i+1} + (1-t)a_i&
a_i&
a_{i+2}&
\dots&
a_n
\end{pmatrix}^\intercal\\
&=
\begin{pmatrix}
I_{i-1} & 0 & 0 & 0\\
0 & 1-t & t & 0\\
0 & 1 & 0 & 0\\
0 & 0 & 0 & I_{n-i-1}
\end{pmatrix}
\begin{pmatrix}
a_1\\
\vdots\\
a_n
\end{pmatrix}\\
&= 
U_i\overline{a}\\
&= 
\psi_n(\sigma_i)(\overline{a}).
\end{align*}
Hence, $\varphi_{\sigma_i} = \psi_n(\sigma_i)$
for all braid group generators $\sigma_i$. 
Consequently, since both $\varphi$ and $\psi_n$ are contravariant functors 
$\B_n \to \mathrm{GL}_n(\Lambda)$, it follows that $\varphi = \psi_n$.
\end{example}

\section{Switch Braid Quivers}\label{BB}

We begin this section with our main definition. Given a switch structure on a
set $X$, we will assign to each braid diagram $B$ a quiver (i.e.,
a directed multi-graph) with $|X|^n$ vertices, such that this assignment is 
invariant under the braid 
Reidemeister moves.

\begin{definition}
Let $X$ be a switch and $B$ an $n$-strand braid diagram. The \textit{switch
braid quiver} associated to $B$ and $X$, denoted $\mathcal{SQ}_X(B)$, has 
\begin{itemize}

\item vertex set $X^n$, and  

\item a directed edge from $\overline{x}=(x_1,\dots, x_n)$ to $\overline{y}=(y_1, \dots, y_n)$ 
if and only if $\varphi_B(\overline{x}) = \overline{y}$.
\end{itemize}
\end{definition}


Note that braid Reidemeister moves do not change the colors at the boundary 
(top and bottom rows) of a
switch-colored braid diagram. Hence, if $B,B'$ are braid diagrams related by a 
braid Reidemeister move, then 
\[\mathcal{SQ}_{X}(B) = \mathcal{SQ}_{X}(B').\] 
Consequently, we have the following result:

\begin{theorem}
For any switch $X$, the function 
\[\mathcal{SQ}_{X}: \{\emph{diagrams of $n$-strand braids}\} \to \{\emph{quivers}\}\] 
sending $B \mapsto \mathcal{SQ}_{X}(B)$ is an invariant of braids.
\end{theorem}

\begin{remark}
We can now treat $\mathcal{SQ}_{X}$ as a function 
$\B_n \to \{\text{quivers}\}$.
\end{remark}

For each braid $B$, $\mathcal{SQ}_{X}(B)$ is a small category which determines
the switch braid counting invariant, and hence is a categorification.
Note that the subquiver of loop edges (edges for which the initial and 
terminal vertices are the same) in $\mathcal{SQ}_{X}(B)$ 
corresponds to colorings of the braid closure $\overline{B}$.


\begin{example}\label{ex1}
Consider the braid 
\[\includegraphics{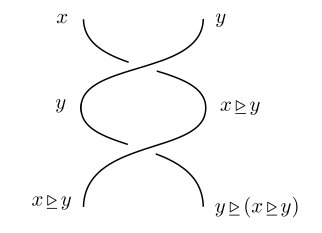}\] 
whose closure is the Hopf link and let $X$ be the birack (in fact, it is a biquandle) with operation tables
\[
\begin{array}{r|rrr}
\utr & 1 & 2 & 3 \\ \hline
1 & 1 & 3 & 2 \\
2 & 3 & 2 & 1 \\
3 & 2 & 1 & 3
\end{array}
\quad 
\begin{array}{r|rrr}
\otr & 1 & 2 & 3 \\ \hline
1 & 1 & 1 & 1 \\
2 & 2 & 2 & 2 \\
3 & 3 & 3 & 3
\end{array}
\]
The set $X^2$ has $|X|^2=9$ elements; we compute that the pair $(x,y)$ is sent
by the braid to the element 
\[(x\utr y,y\utr(x\utr y))\]
and thus obtain switch braid quiver
\[\includegraphics{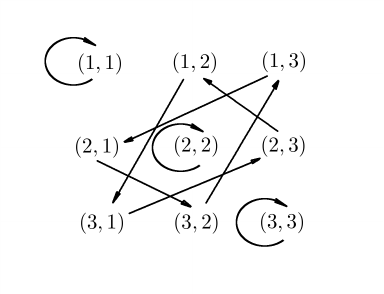}\] 
\end{example}

%
%
%
%

Since $\varphi_B$ is a bijection for every $B$, we have:

\begin{theorem}
For any braid $B$ and finite switch $X$, the switch braid quiver 
$\mathcal{SQ}_{X}(B)$ decomposes into disjoint directed cycles. 
\end{theorem}

For each vertex $v$ in $\mathcal{SQ}_{X}(B)$, let $\mathcal{L}(v)$ denote the 
length of the directed cycle containing $v$. 
Then we make the following 
definition:

\begin{definition}
Let $X$ be a finite switch and $B$ a braid diagram. We define the 
\textit{switch braid quiver polynomial} of $B$ with respect to $X$ to be
the polynomial obtained by summing over the 
vertices $v$ in $\mathcal{SQ}_X(B)$:
\[\Phi_X^{C}(B)=\sum_{v\in \mathcal{SQ}_X(B)} u^{\mathcal{L}(v)}.\]
\end{definition}

Since the switch braid quiver is an invariant of braids, it follows that:

\begin{corollary}
The switch braid quiver polynomial is an invariant of braids.
\end{corollary}

\begin{example}\label{ex2}
The braid in Example \ref{ex1} has switch braid quiver invariant
value $\Phi_X^{C}(B)=6u^3+3u$ with respect to the birack in the example.
\end{example}

Since a coloring with switch braid quiver cycle length 1 is a coloring
of the braid closure, we obtain:

\begin{corollary}
When $X$ is a birack, the coefficient of the linear term 
$u$ in $\Phi_X^{C}(B)$ is the birack
counting invariant of the braid closure of $B$, 
namely $\Phi_X^{\mathbb{Z}}\left(\overline{B}\right)$.
\end{corollary}

%
%
%

This last corollary poses a natural question: What is the meaning of the 
other coefficients? 
It is not difficult to see that 
a coloring with cycle length $k$ is a coloring of 
the closure of $B^k$, the $k$th power of $B$ in the braid group; hence we have
the more general result:

\begin{corollary}
When $X$ is a birack, the coefficient of the term $u^k$ in $\Phi_X^{C}(B)$ is the birack 
counting invariant of the braid closure of $B^k$, namely 
$\Phi_X^{\mathbb{Z}}\left(\overline{B^k}\right)$.
\end{corollary} 

\begin{example} 
The invariant value $\Phi_X^{C}(B)=6u^3+3u$ from example \ref{ex2} says that 
closure of the braid in Example \ref{ex1}, i.e. the Hopf link, has three 
colorings by the birack $X$ in the example, distinguishing it from the 
two-component unlink which has nine. Moreover, it also says that the closure 
of the third power of the braid, i.e. the $(6,2)$-torus link, has six colorings 
by $X$. 
\end{example}

\section{Switch Braid Quiver Cocycle Enhancements}\label{BCQ}

The meanings of the individual terms of the switch braid quiver 
polynomial all have been identified. However, like the counting invariants
from which the polynomial arises, this invariant can be enhanced in various 
ways. In this section we will consider enhancements via birack 2-cocycles,
leaving further enhancements for future work.

Let $X$ be a finite birack and $A$ an abelian group. Recall (see \cite{EN}
for more detail) that the group of \textit{birack $k$-chains} is
the $A$-module generated by ordered $k$-tuples of elements of $X$, i.e.
$C_k(X;A)=A[X^k]$. The map $\partial_k:C_k(X;A)\to C_{k-1}(X;A)$ defined on
generators by
\[\partial_k(x_1,\dots, x_k)
=\sum_{j=1}^k(-1)^j[(x_1,\dots, x_{j-1},x_{j+1},\dots, x_k) -
(x_1\utr x_j,\dots, x_{j-1}\utr x_j,x_{j+1}\otr x_j,\dots, x_k\otr x_j)]\]
and extended by linearity is known as the \textit{birack boundary map}. 
For each $k$ we have the groups of \textit{boundaries} 
\[B_{k+1}(X;A)=\mathrm{Im}\,\partial_{k+1}\subset C_k(X;A)\]
and \textit{cycles}
\[Z_k(X;A)=\mathrm{Ker}\,\partial_k\subset C_k(X;A)\]
with the quotient group giving us the \textit{$k$th birack homology} 
\[H_k(X;A)=Z_k(X;A)/B_{k+1}(X;A).\]
Dualizing yields the corresponding cohomology modules, 
$H^k(X;A)=Z^k(X;A)/B^{k-1}(X;A)$.

Given an element of $Z^2(X;A)$ representing a cohomology class in $H^2(X;A)$, 
we collect contributions from the crossings in a birack coloring of $B$ given by
\[\includegraphics{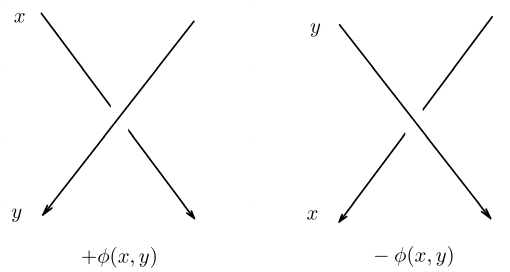}\]
to obtain the \textit{Boltzmann weight} $BW(v)$ of the coloring vector $v \in X^n$
(recalling that the vector of colors at the top determines the entire coloring). 
It is a standard observation (see \cite{EN} for details) that this Boltzmann
weight is unchanged by Reidemeister II and III moves and hence is a braid
invariant. Moreover, cohomologous cocycles define the same invariant.
We can thus enhance the switch braid quiver $\mathcal{SQ}_{X}(B)$ by selecting
a birack 2-cocycle $\phi\in H^2(X;A)$ for a coefficient group $A$ and
assigning a weight of $BW(v)$ to each vertex (a.k.a. coloring vector) 
in $\mathcal{SQ}_{X}(B)$.
More formally, we have:

\begin{definition}
Let $B$ be a braid diagram, $X$ a finite birack and $A$ an abelian group. For each 
element $\phi\in H^2(X;A)$, we define the \textit{birack braid cocycle quiver}
$\mathcal{BCQ}_X^{\phi}(B)$
to be the switch braid quiver $\mathcal{SQ}_{X}(B)$ with vertices weighted
with the Boltzmann weights: 
\[BW(v)=\sum_{c\in B}(-1)^{\epsilon(c)}\phi(x(c),y(c))\]
where we sum over crossings $c$ in $B$, with
$\epsilon(c)$ denoting the sign of crossing $c$ and $(x(c),y(c))$
denoting the tuple of colors on the left of crossing $c$ induced by the 
color vector $v$, as in the diagram above.
\end{definition}

%
%
%
%

By construction, we have:

\begin{theorem}
For any finite birack $X$, abelian group $A$ and birack 2-cocycle $\phi\in H^2(X;A)$, the function 
\[\mathcal{BCQ}_{X}^{\phi}: \{\emph{diagrams of $n$-strand braids}\} \to \{\emph{vertex-weighted quivers}\}\] 
sending $B \mapsto \mathcal{BCQ}_X^{\phi}(B)$ is an invariant of braids.
\end{theorem}

\begin{remark}
We can now treat $\mathcal{BCQ}_{X}^{\phi}$ as a map $\B_n \to  \{\text{vertex-weighted quivers}\}$.
\end{remark}

\begin{example}\label{ex4}
Let $X=\{1,2,3\}$ have the constant action birack structure given by 
$\sigma=\tau=(13)$ and let $\phi:X\times X\to \mathbb{Z}_5$ be given by 
\[\phi(x,y)=\chi_{(1,2)}+4\chi_{(1,3)}+3\chi_{(2,1)}+2\chi_{(2,3)}+\chi_{(3,1)}+2\chi_{(3,2)}.\] The reader can verify that $\phi$ is a birack 2-cocycle.
Then, for example, the $X$-coloring of the braid $B=\sigma_1^2\in B_2$ 
determined by the top row vector $(2,3)$ shown has Boltzmann weight 
$\phi(2,1)+\phi(1,2)+\phi(2,1)= 3+1+3=2$:
\[\includegraphics{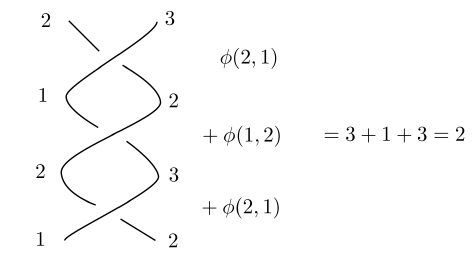}\]
Repeating over the set $X^2$, $B$ has birack braid cocycle quiver
\[\includegraphics{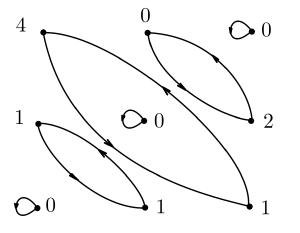}\]
\end{example}

\begin{definition}
For any braid diagram $B$, finite birack $X$, abelian group $A$ and birack 2-cocycle $\phi\in C^2(X;A)$, 
make the following definitions.
\begin{itemize}
\item The \textit{birack braid cocycle quiver polynomial} 
$\Phi_X^{\mathcal{BCQ},\phi}(B)$ is the sum over the vertices $v$ in 
$\mathcal{BCQ}_X^\phi(B)$: 
\[\Phi_X^{\mathcal{BCQ},\phi}(B)=\sum_{v\in \mathcal{BCQ}_X^\phi(B)}u^{\mathcal{L}(v)}v^{BW(v)}\]
where the terms encode the cycle length and Boltzmann weight of each vertex. 
\item The \textit{birack braid cocycle 2-variable polynomial} is the
sum over the edges $e$ in $\mathcal{BCQ}_X^\phi(B)$:
\[\Phi_X^{\phi,2}(B)=\sum_{e\in \mathcal{BCQ}_X^\phi(B)}s^{BW(s(e))}t^{BW(t(e))}\]
where
$BW(s(e))$ is the Boltzmann weight of the source vertex of $e$ and
$BW(t(e))$ is the Boltzmann weight of the target vertex of $e$.
\end{itemize}
\end{definition}

Since both polynomials are decategorifications of the braid invariant
$\mathcal{BCQ}_X^{\phi}$, we have:

\begin{corollary}
Both polynomials $\Phi_X^{\mathcal{BCQ},\phi}$ and $\Phi_X^{\phi,2}$
are invariants of braids.
\end{corollary}

\begin{example}
The braid in Example \ref{ex4} has 
\[\Phi_X^{\mathcal{BCQ},\phi}(B)=u^2v^4+u^2v^2+3u^2v+u^2+3u\]
and
\[\Phi_X^{\phi,2}(B)=s^4t+s^2+st^4+2st+t^2+3.\]
\end{example}

\section{Questions}\label{Q}

We conclude with some questions for future research.

\begin{itemize}
\item What other enhancements of $\Phi_{X}^{C}$ are possible?
\item Beyond replacing with biracks with biquandles, what strategies 
are needed to obtain invariants of knots and link from these invariants 
of braids?
\item Is there a natural notion of an inner automorphism of a switch?
If so, which braids $B$ make $\varphi_B$ an inner automorphism of $X^n$?
\end{itemize}


\bibliography{mch-sn}{}

\begin{thebibliography}{1}

\bibitem{CN}
K.~Cho and S.~Nelson.
\newblock Quandle coloring quivers.
\newblock {\em Journal of Knot Theory and Its Ramifications}, 28(01), 2018.

\bibitem{SCN}
A.~S. Crans and S.~Nelson.
\newblock Hom quandles.
\newblock {\em Journal of Knot Theory and Its Ramifications}, 23(02), 2 2014.

\bibitem{DD}
B.~A. Davey and G.~Davis.
\newblock Tensor products and entropic varieties.
\newblock {\em Algebra Universalis}, 21:68--88, 1985.

\bibitem{EN}
M.~Elhamdadi and S.~Nelson.
\newblock {\em Quandles---An Introduction to the Algebra of Knots}, volume~74
  of {\em Student Mathematical Library}.
\newblock American Mathematical Society, Providence, RI, 2015.

\bibitem{FJK}
R.~Fenn, M.~Jordan-Santana, and L.~Kauffman.
\newblock Biquandles and virtual links.
\newblock {\em Topology and its Applications}, 145(1-3):157--175, 2004.

\bibitem{FRS}
R.~Fenn, C.~Rourke, and B.~Sanderson.
\newblock Trunks and classifying spaces.
\newblock {\em Applied Categorical Structures}, 3(4):321--356, 1995.

\bibitem{SCH}
E.~Horvat and A.~S. Crans.
\newblock From biquandle structures to hom-biquandles.
\newblock {\em Journal of Knot Theory and Its Ramifications}, 29(02), 2 2020.

\bibitem{JPSZ}
P.~Jedli{\v c}ka, A.~Pilitowska, D.~Stanovsk{\'y}, and A.~Zamojska-Dzienio.
\newblock The structure of medial quandles.
\newblock {\em Journal of Algebra}, 443:300--334, 12 2015.

\bibitem{KR}
L.~H. Kauffman and D.~Radford.
\newblock Bi-oriented quantum algebras, and a generalized {A}lexander
  polynomial for virtual links.
\newblock In {\em Diagrammatic Morphisms and Applications}, volume 318 of {\em
  Contemp. Math.}, pages 113--140. Amer. Math. Soc., Providence, RI, 2003.

\end{thebibliography}
\bibliographystyle{abbrv}

\bigskip

\noindent
\textsc{Department of Mathematics \\
Harvey Mudd College\\
301 Platt Boulevard \\
Claremont, CA 91711
}

\bigskip

\noindent
\textsc{Department of Mathematical Sciences \\
Claremont McKenna College \\
850 Columbia Ave. \\
Claremont, CA 91711}

\end{document}